\documentclass[leqno]{article}
\usepackage{graphicx}	
\usepackage[cp1251]{inputenc}
\usepackage[english]{babel}
\usepackage{mathtools}
\usepackage{amsfonts,amssymb,mathrsfs,amscd,amsmath,amsthm}

\usepackage{verbatim}

\usepackage{url}

\usepackage{stmaryrd}

\def\ig#1#2#3#4{\begin{figure}[!ht]\begin{center}%
\includegraphics[height=#2\textheight]{#1.pdf}\caption{#4}\label{#3}%
\end{center}\end{figure}}

\def\thtext#1{
  \catcode`@=11
  \gdef\@thmcountersep{. #1}
  \catcode`@=12
}

\def\threst{
  \catcode`@=11
  \gdef\@thmcountersep{.}
  \catcode`@=12
}

\theoremstyle{plain}
\newtheorem{thm}{Theorem}[section]

\newtheorem{ass}[thm]{Assertion}
\newtheorem{lem}[thm]{Lemma}

\theoremstyle{definition}
\newtheorem{rk}[thm]{Remark}

 \catcode`@=11
 \def\.{.\spacefactor\@m}
 \catcode`@=12

\def\N{{\mathbb N}}

\def\R{\mathbb R}

\def\Z{{\mathbb Z}}

\def\g{\gamma}

\def\om{\omega}
\def\Om{\Omega}

\def\l{\lambda}
\def\L{\Lambda}
\def\r{\rho}

\def\0{\emptyset}
\def\:{\colon}
\def\<{\langle}
\def\>{\rangle}
\def\[{\llbracket}
\def\]{\rrbracket}

\def\rom#1{\emph{#1}}
\def\({\rom(}
\def\){\rom)}

\def\ss{\subset}

\def\x{\times}

\def\mpf{\operatorname{mpf}}

\def\mf{\operatorname{mf}}

\def\cG{{\mathcal G}}

\def\cM{{\mathcal M}}

\def\cP{{\mathcal P}}

\def\sss#1{{\scriptscriptstyle #1}}

\begin{document}
\title{Dual Linear Programming Problem and One-Dimensional Gromov Minimal Fillings of Finite Metric Spaces}
\author{A.\,O.~Ivanov, A.\,A.~Tuzhilin}
\date{}
\maketitle

\begin{abstract}
The present paper is devoted to studying of minimal parametric fillings of finite metric spaces (a version of optimal connection problem) by methods of Linear Programming. The estimate on the multiplicity of multi-tours appearing in the formula of weight of minimal fillings is improved, an alternative proof of this formula is obtained, and also explicit formulas for finite spaces consisting of $5$ and $6$ points are derived.
\end{abstract}

\section*{Introduction}
\markright{\thesection.~Introduction}
The range of Optimal Connections Problem is rather wide. Informally speaking, a given subset $M$ (usually a finite one) of a metric space $X$ has to be connected by some network (a connected graph) of an optimal length. As the concept of optimality, so as the class of admissible graphs could be defined in different ways. For example, the \emph{Steiner Problem\/} consists in finding a connected graph having the least possible length and such that its vertex set is contained in $X$ and contains the initial set $M$ (in the latter case we say that this graph \emph{connects $M$}). The length of the graph is defined as the sum of lengths of all its edges, and the length of an edge is equal to the distance in $X$ between its vertices. A solution to this problem is referred as a \emph{shortest network\/} on $M$.

The concept of minimal filling of a finite metric space $M$ appeared recently in paper~\cite{IvaTuzhMF} by the authors as a generalisation of the Steiner Problem and a concept of minimal filling of Riemannian manifold in the sense of  M.~Gromov~\cite{Gromov}. Recall, see also Section~\ref{subsec:fill}, that a tree $G=(V,E)$ with an arbitrary weight function $\om$ given on its edge set is called a \emph{filling of the type $G$ of the space $M$}, if $M\ss V$ as a subset, and for any pair of elements from $M$ the weight of the path in $G$ connecting those vertices (i.e., the sun of $\om(e)$ over all edges $e$ belonging to the path) is greater than or equal to the distance between these points in the metric space $M$. A filling of the type $G$ having the least possible weight is called a \emph{minimal parametric filling of the type $G$}. Passing to the infimum over all possible types $G$ leads to the concept of \emph{minimal filling}. Minimal Fillings Theory of finite metric spaces is actively evolved, see for example~\cite{Trudy}, \cite{Eremin}, \cite{Ovs}, and~\cite{Bed}.  It turns out to be related not only with geometry of finite metric spaces, but with structure of optimal networks also. For example, the weight of a minimal filling of a set $M\ss X$ considered as a finite metric space with the metric induced from  $X$ gives an exact lower bound for the length of shortest network connecting $M$ in $X$.

As it was already noticed in~\cite{IvaTuzhMF} by the authors, the problem of finding a minimal parametric filling of a given type can be reduced to a linear programming, but this fact was used in the proof of the corresponding existence theorem only.  However Linear Programming is one of the most developed and widely used branches of general optimization theory, it has many useful methods and approaches, and  one of them is duality principal. In the case of minimal filling the passing to the dual problem, see below, turns out to be rather effective.

Roughly speaking, by each binary tree $G$ with $n$ vertices of degree $1$ we construct a convex polyhedron of dimension $(n-2)(n-3)/2$ in the space of dimension $n(n-1)/2$, and finding minimal parametric filling of the type $G$ is reduced to maximization on this polyhedron of a linear function which is defined in terms of the metric on $M$. The maximal value is achieved at some vertex of the polyhedron, and the problem can be reduced to checking of its vertex set. The maximal value of the function turns out to be a rational linear combination of some distances on $M$, that defines a multi-tour of $G$ and its multiplicity.

In the present paper this method is used to obtain an alternative proof of the non-trivial formula of the weight of a parametric minimal filling obtained by A.~Eremin in~\cite{Eremin} in terms of multi-tours by a rater tricky technique, see below, and essentially improve the estimate on the multiplicity of the multi-tours standing in the formula. We also derived explicit formulas for the weight of a minimal parametric fillings of any fixed type for five-point and six-point metric spaces  (the case of five-point metric spaces was considered in~\cite{Bed}, and the case of six-point spaces is a new one) and, in fact, we give an algorithm permitting to obtain such formulas for any finite metric space and any fixed type of a filling.

\section{Preliminaries}
\markright{\thesection.~Preliminaries}
In the present Section we introduce necessary notations and concepts and recall basic facts from Minimal Fillings of Finite Metric Spaces and Linear Programming Theories. More detailed information on the minimal fillings can be found in~\cite{IvaTuzhMF}, and in a review~\cite{IT_Musin}. From a huge number of books on Linear Programming we use a more geometrical book~\cite{VasIvanits}.

\subsection{Minimal Fillings of Finite Metric Spaces}
\label{subsec:fill}

Let $M$ be an arbitrary finite set and let $G=(V,E)$ be some connected simple graph with a vertex set $V$ and an edge set $E$. To be short, the edge of the graph connecting its vertices $u$ and $v$ is denoted by $uv$. We say that $G$ {\em connects $M$}, if $M\subset V$. In this case we also say that $M$ is a {\em boundary of the graph $G$}. In what follows we always assume that each graph has some fixed boundary which could be empty. A \emph{tree\/} is a connected graph without cycles, whose boundary contains all its vertices of degree $1$ and $2$. A tree is said to be  \emph{binary}, if the degrees of its vertices are equal to $1$ or to $3$, and its boundary coincides with the set of all its vertices of degree $1$.

A \emph{cut of a graph $G$} is an arbitrary partition of its vertex set into two non-empty non-intersecting subsets. An edge is called a \emph{cut edge of the cut $V=V_1\sqcup V_2$},  if one of its vertices belongs to $V_1$, and another vertex belongs to $V_2$. For any connected graph the set of cut edges is non-empty for any cut. To each family of cuts the so-called \emph{cut matrix} corresponds, whose rows are enumerated by cuts of the family, whose columns are enumerated by the edges of the graph, and whose element standing at the $j$th position in the $i$th row equals $1$, if the $j$th edge is a cut edge of the $i$th cut, and equals $0$ otherwise.

Let $G$ be an arbitrary tree with a boundary $M$ consisting of all the vertices of $G$ of degree  $1$ and $2$, and  let $e\in E$ be an arbitrary edge of the tree $G$. Elimination of the edge $e$ partitions the tree $G$ into two connected components that are denoted by $G_1$ and $G_2$.Put $M_i = M\cap G_i$, $i=1,\,2$. By $\cP_G(e) = \{M_1,\,M_2\}$ we denote the resulting partition of the set $M$. In particular, each edge of the tree $G$ generates a cut of the complete graph  $K(M)$ with the vertex set $M$.

Now let $\cM=(M,\rho)$ be a finite pseudo-metric space (recall that in a pseudo-metric space the distances between distinct points could be equal to zero), $G=(V,E)$ be a connected graph connecting $M$, and $\om\:E\to\R_+$ be a mapping of its edge set into non-negative reals ,that is usually referred as a {\em weight function}. The pair $\cG=(G,\om)$ is called a {\em weighted graph}. The {\em weight  $\om(\cG)$} of a weighted graph $\cG$ is the sum of weights $\om(e)$ of all edges $e$ of the graph. The function $\om$ generates a pseudo-metric $d_\om$ on $V$ as follows: the $d_\om$-distance between vertices of the graph $\cG$ is defined as the least weight of the walks connecting those vertices.  If for any pair of points $p$ and $q$ from  $M$ the inequality $\rho(p,q)\le d_\om(p,q)$ holds, then the weighted graph $\cG$ is called a {\em filling\/} of the space $\cM$, and the graph $G$ is called the {\em type\/} of this filling. The value $\mf(\cM)=\inf\om(\cG)$, where the infimum is taken over all fillings $\cG$ of the space $\cM$ is called the {\em weight of minimal filling}, and a filling $\cG$ such that $\om(\cG)=\mf(\cM)$ is called a  {\em minimal filling}. If one minimizes the weight of fillings over fillings of a fixed type, then one obtains a {\em minimal parametric fillings}, whose weight is denoted by $\mpf(\cM,G)$. In paper~\cite{IvaTuzhMF} it is shown that it suffices to restrict the search of minimal fillings by the fillings, whose type is a tree without vertices of degree $2$, and such that its set of vertices of degree $1$ coincides with $M$.

It turns out that it is not necessary to define fillings among the graphs with non-negative weight functions only, namely, one can find the minimum over fillings with arbitrary weight functions, not necessary nonnegative ones. Such neighborhoods are referred as {\it generalised\/} ones. It turns out that under such switch to generalized minimal fillings the most of properties of minimal fillings are preserved. Let us pass to formalities, see details in~\cite{NegativeFillings}.

By a {\it generalized weighted graph\/} we call a pair $(G,\om)= (V,E,\om)$, where $\om\colon E\rightarrow \R$ is an arbitrary function. Define $d_{\om}\colon V\times V \rightarrow \mathbb{R}$ as follows: $d_{\om}(u,v)$ equals the least possible weight of the paths  (i.e., walks with pairwise distinct edges) with the ends $u$ and $v$. Generally speaking, the function $d_{\om}$ is not non-negative, and need not satisfy the triangle inequalities.

By a {\it generalized filling of a finite pseudo-metric space $\cM =(M,\r)$} we call a generalized weighted tree $\cG$ connecting~$M$, if for any $u,v\in M$ the inequality $\r(u,v)\leqslant d_{\om}(u,v)$ holds.

\begin{rk}
Of course, in any tree any two vertices are connected by a unique path that simplifies the definition of $d_\om$. The concept of generalized filling can be given for an arbitrary connected graph connecting $M$, see~\cite{NegativeFillings}, but in this paper it suffices to consider the trees.
\end{rk}

The {\it weight $\mpf_{-}(\cM,G)$ of generalized minimal parametric filling of a type $G$ of a space $\cM$} is defined as $\inf \om({\cG})$, where the infimum is taken over all generalized fillings of the space $\cM$ of the given fixed type $G$. Each generalized filling of the type $G$, which the infimum is attained at is called a {\it generalized minimal parametric filling of the type  $G$ of the space $M$}.

The {\it weight $\mf_{-}(\mathcal{M})$ of generalized minimal filling\/} is defined as the value $\inf \mpf_{-}({\cM}, G)$, where the infimum is taken over all trees $G$ connecting $M$. Each generalized filling which the infimum is attained at is called a {\it generalized minimal filling of $M$}.

A remarkable fact is that for any finite pseudo-metric space the weights of minimal filling and of generalized minimal filling are equal to each other, i.e. the minimal weight over the fillings with non-negative weight functions coincides with the minimal weight over the larger set of fillings with arbitrary weight functions. Namely, the following result holds, see~\cite{NegativeFillings}.

\begin{thm}\label{thMFPEqualsMFN}
For any finite pseudo-metric space ${\cM}$ the equality ${\rm mf_{-}}({\cM})={\rm mf}({\cM})$ holds.
\end{thm}

This result permits essentially simplify the initial problem of minimal fillings finding, because the necessity to verify non-negativity of the weights disappears.

\begin{rk}
A similar statement for minimal parametric fillings is not valid. It is not difficult to construct an example of a space $\cM$ and a tree $G$, such that ${\mpf_{-}}({\cM},G)<{\rm mpf}({\cM},G)$.
\end{rk}

\subsection{Linear Programming}\label{ssec:lp}

Recall the statement of so-called \emph{General  Linear Programming Problem\/} that is abbreviated to \emph{GLPP}. Consider $n$-dimensional linear space that is convenient to represent as $\R^n=\R^{n_1}\times\R^{n_2}$. Write vectors from $\R^n$ in the form  $x=(x_1,x_2)$, where $x_i\in\R^{n_i}$. Let we be given with a linear function $F(x)=\<f_1,x_1\>+\<f_2,x_2\>$, where $f_i\in\R^{n_i}$, and angle brackets stand for the standard scalar product. The problem is: Find the least possible value of the function $F$ (in this context $F$ is referred as \emph{objective function\/}) on the subset $X\subset\R^n$ of the space $\R^n$, where $X$ is defined by a system of linear equations and linear equalities as follows:
$$
X=\big\{x=(x_1,x_2)\mid A_{11}x_1+A_{12}x_2\le b_1,\quad
A_{21}x_1+A_{22}x_2=b_2,\quad x_1\ge0\big\},
$$
where $A_{ij}$  are fixed  $(m_i\times n_j)$ matrices, and  $b_i\in\R^{m_i}$. The set $X$ is called the \emph{solutions space}. If $X$ is not empty, then $X$ is a convex polyhedral subset of the space $\R^n$. In this case we put $F_*=\inf_{x\in X}F(x)$, and if $F_*$ is finite, then $X_*=\{x\in X:F(x)=F_*\}$. The problem is called \emph{solvable}, if $X_*$ is non-empty. In this case each point $x_*\in X_*$ is called a  \emph{solution}. Notice that one can search the greatest possible value of $F$ on $X$ instead of the least one, and one problem can be reduced to the other by changing of the sign of the objective function $F$. The supremum of the values $F(x)$ over $X$ is denoted by $F^*$, and if $F^*$ is finite, then $X^*=\{x\in X:F(x)=F^*\}$.

Consider an important articular case that is referred as  \emph{Canonical Linear Programming Problem}, or \emph{CLLP}. In the above notations the CLLP corresponds to the case  $n_2=m_1=0$, i.e., all the variables are non-negative, and all the remaining constrains are equations (it is well-known that each GLPP can be reduced to some CLLP by a standard trick with addition of new non-negative variables).

The solution set of any CLLP is contained in the positive orthant. Therefore, if CLPP is solvable, then its set of solutions $X$ always contains so-called  \emph{angular\/} or \emph{extreme points\/} of $X$. Recall that a point of $X$ is said to be extreme, if it does not lie inside any interval of non-trivial segment, whose end points belong to $X$. Moreover, in this case the set  $X_*$ is a convex polyhedral set all whose angular points is also angular points of $X$.  From geometrical point of view, these angular points are vertices of the polyhedral set.

Angular points of the solution set $X=\{x\in\R^n\mid x\ge0,\quad Ax=b\}$ of a canonical problem can be easily described (here for brevity $n=n_1$, $b=b_2$, $A=A_{21}$). By $A_j$ we denote the $j$th column of the matrix $A$.

\begin{ass}\label{ass:extreme_points}
Under the above notations, let $A\ne0$, and let $r$ be the rank of the matrix $A$. A point  $x=(x^1,\ldots,x^n)\in\R^n$ is angular, if and only if there exist linear independent columns  $A_{j_1},\ldots,A_{j_r}$ of the matrix $A$, such that
$$
A_{j_1}x^{j_1}+\cdots+A_{j_r}x^{j_r}=b,
$$
and $x^{j_k}\ge 0$, $k=1,\ldots,r$,  and the remaining $x^j$ equal $0$.
\end{ass}

In Linear Programming a duality principal plays an important role. Recall that the \emph{dual problem\/} to GLLP is stated as follows: Find the greatest possible value of the linear function  $H(\l)=-\<b_1,\l_1\>-\<b_2,\l_2\>$, where the variable vectors $\l_i\in\R^{m_i}$ form the vector  $\l=(\l_1,\l_2)$ belonging to the polyhedral domain $\L\ss\R^m$, $m=m_1+m_2$, that is defined by the following system of linear equations and inequalities:
$$
\L=\big\{\l=(\l_1,\l_2)\mid A_{11}^T\l_1+A_{21}^T\l_2\le - f_1,\quad
A_{12}^T\l_1+A_{22}^T\l_2=-f_2,\quad \l_1\ge0\big\}.
$$
As it is well-known, the dual problem to the dual problem is equivalent to the initial one, therefore usually one speaks about \emph{mutual duality}.

In accordance with \emph{Duality Principle}, mutually dual problems of Linear Programming are solvable or non solvable simultaneously, and if the problems are solvable, then $F_*=H^*$, and  $F(x)=F_*=H^*=H(\l)$, iff $x\in X_*$ and $\l\in\L^*$.

\section{Minimal Parametric Fillings and Linear Programming}
\markright{\thesection.~Minimal Parametric Fillings and Linear Programming}

As it is noticed in~\cite{IvaTuzhMF}, the problem of finding a minimal parametric filling can be reduced to a Linear Programming. Indeed, let $(M,\r)$ be an arbitrary finite metric space. Enumerate points from $M=\{p_1,\ldots,p_n\}$ in some a way, and put $d_{ij}=\r(p_i,p_j)$. Further, let $G=(V,E)$ be some tree connecting $M$. Assume that $M$ coincides with the set of vertices of the tree $G$ of degrees $1$ and $2$. Describe weight functions $\om\:E\to\R$ that make the tree $G$ to a generalized filling of the finite metric space $(M,\r)$. For each pair $p_i$, $p_j$ of the boundary vertices there exists unique path $\g(i,j)$ in $G$ connecting them. By definition, a weighted tree $(G,\om)$ is a generalized filling of $(M,\r)$, if and only if
\begin{equation}\label{eq:omega}
\sum_{e\in\g(i,j)}\om(e)\ge d_{ij},\quad\text{for all $1\le i<j\le n$}.
\end{equation}
The weight of this filling equals $\sum_{e\in E}\om(e)$, where $E$ is the edge set of the tree $G$. Thus, to find a minimal parametric filling of the type $G$, it is necessary to find the least value of the linear function $F(\om)=\sum_{e\in E}\om(e)$ on the convex polyhedral subset $\Om_G$ of the space $\R^{|E|}$ defined by the system of linear inequalities~$(\ref{eq:omega})$. Write down this Linear Programming Problem as a GLPP as it is described in Section~\ref{ssec:lp}.

Since we have no restriction on the signs of the variables, then $n_1=0$, and $n_2$ is equal to the number $|E|$ of edges of the tree $G$. The variables forming the vector $x_2$ are the variables $\om(e)$. Enumerate the edges of the tree in some a way and put $E=\{e_1,\ldots,e_{|E|}\}$, $\om(e_i)=\om_i$. Further, all our constrains on the variables $\om_i$ have the form of inequalities, therefore only the matrix $A_{12}$ is non-zero. The rows of this matrix correspond to the inequalities from~$(\ref{eq:omega})$, i.e., they are enumerated by the ordered pairs $(i,j)$, $1\le i<j\le n$, where $M$, the set of boundary vertices of the tree $G$, consists of $n$ elements as above. Therefore $m_1=n(n-1)/2$. The columns of the matrix $A_{12}$ corresponds to edges of the tree. By $a_{ij}^k$ we denote the element of the matrix $A_{12}$ standing at the row $(i,j)$ at the place corresponding to the edge $e_k$. And $a_{ij}^k=1$, if and only if the edge $e_k$ belongs to the path $\g(i,j)$, otherwise $a_{ij}^k=0$. The subset  $\Om_G$ is defined by the inequalities system
$$
A_{12} x_2\le b_1,\quad \text{where $x_2=-(\om_1,\ldots,\om_{|E|})$, and $b_1=-(d_{12},\dots,d_{(n-1)\,n})$},
$$
and we are looking for the least possible value of the linear function $F(x)=\<f_2,x_2\>$ on $\Om_G$, where $f_2=-(1,\ldots,1)$.

Now, let us write down the dual problem. In our case $m_2=0$, and the components $\l_{ij}$ of the $m_1$-dimensional vector $\l_1$are enumerated by the ordered pairs$(i,j)$, $1\le i<j\le n$. The dual linear function has the form
$$
H(\l)=H(\l_1)=-\<b_1,\l_1\>=\sum_{1\le i<j\le n} d_{ij}\l_{ij}.
$$
The solution space $\L_G$, i.e., the subset of the space $\R^{m_1}$, which the greatest value of the function $H$ is searched on, is defined by the following system of linear constrains:
\begin{equation}\label{eq:dual}
A_{12}^T\l_1=-f_2, \qquad \l_1\ge 0.
\end{equation}
In particular, the dual problem is a CLPP. The matrix $A_{12}^T$ is an $(n_2\times m_1)$ matrix. Its rows correspond to edges of the tree $G$, and its columns correspond to ordered pairs $(i,j)$, $i<j$, which can be naturally interpreted as the edges of the complete graph $K(M)$ with the vertex set $M$. By $a^{ij}_k$ we denote the element of the $k$th row of the matrix $A_{12}^T$ standing at the $(i,j)$th column. Let the edge $e_k$ of the tree $G$ generates the partition $\cP_G(e_k)$ of the set $M$ into two non-empty subsets. The element $a^{ij}_k$ equals $1$, if and only if the vertices of the edge $(i,j)$ of the graph $K(M)$ belong to distinct elements of the partition $\cP_G(e_k)$. All the components of the vector $-f_2$ equal $1$.

\begin{ass}\label{ass:mat_cut}
The matrix $A_{12}^T$ is the cut matrix of the graph $K(M)$ corresponding to the cuts family generated by the edges of the tree $G$ connecting $M$.
\end{ass}

In what follows we are interested in the case of binary trees only. If the boundary of a binary tree consists of $n\ge2$ elements, then the number of its vertices of degree $3$ is equal to  $n-2$, and the number of edges equals $2n-3$. If $G$ is a binary tree, and $M$ is the set of its boundary vertices, then by  $C(G)$ we denote the cut matrix of the complete graph $K(M)$, corresponding to the cuts generated by the edges of $G$.

Recall that by \emph{moustaches of a binary tree $G$} we call a pair of its boundary vertices having a common neighbouring vertex.  Also by moustaches we call the corresponding pair of adjacent edges that are incident to those boundary vertices--moustaches. It is easy to verify that any binary tree with $n\ge3$ boundary vertices has moustaches. The operation of  \emph{moustaches elimination\/} can be naturally defined as follows: from the vertex set and the edge set of the binary tree we through out, respectively, two vertices and two edges forming some its moustaches (the same ones). The new vertex of degree $1$ of the tree obtained is proclaimed as a boundary one, and the resulting tree becomes a binary one.

\begin{lem}\label{lem:rk}
Let $G$ be an arbitrary binary tree with a boundary $M$ consisting of $n\ge2$ vertices. Then the rank of the matrix $C(G)$ is maximal, and it equals $2n-3$.
\end{lem}

\begin{proof}
Proceed an induction on the number $n$ of boundary vertices of the tree $G$. If $n=2$, then  $G=K(M)$, and $C(G)=(1)$, and hence, the statement of Lemma is true.

Now consider an arbitrary binary tree $G$ with $n+1$ boundary vertices and chose its arbitrary moustaches. Enumerate the boundary vertices of the tree $G$ in such a way that the vertices forming the moustaches get the numbers $n$ and $n+1$, and enumerate the edges of $G$ in such a way that the edges forming the moustaches get the last numbers.

Eliminate the chosen moustaches. Preserve the numeration of edges in the resulting binary tree $H$, and let $n$ be the number of the new boundary vertex. The matrices $C(G)$ and $C(H)$ are related in the following way. As above, enumerate the edges of the complete graph by the pairs of the form $(i,j)$, $i<j$, that are ordered lexicographically. The matrix $C(G)$ has $n$ more columns than $C(H)$; these columns correspond to the pairs $(1,n+1)$, $(2,n+1)$, \dots, $(n-1,n+1)$, and $(n,n+1)$. Besides, the matrix $C(G)$ has two more rows. These rows correspond to the edges of the moustaches chosen in $G$, and they are denoted by $a$ and $b$. We get the matrix of the following form (in notations of the columns the parenthesis are omitted to save the place):
$$
\begin{array}{cc}
&
\begin{array}{cccccccccccccccc}
 \sss{1,2} & \cdot &\cdot&\sss{1,n} &\sss{1,n+1} & \sss{2,1}&\cdot &\sss{2,n}&\sss{2,n+1}&\cdot&\cdot
&\sss{n-1,n}&\sss{n-1,n+1}&\sss{n,n+1}
\end{array} \\
\begin{array}{c}
\\
e\\
\\
\\
a\\
b
\end{array}
&
\left(\begin{array}{cccccccccccccccc}
 \cdot & \cdot &\cdot& X & X & \cdot &\cdot & X & X &\cdot
&\cdot& X &X&0 \\
\cdot & \cdot &\cdot& X & X & \cdot &\cdot & X & X &\cdot
&\cdot& X &X&0 \\
 \cdot & \cdot &\cdot& X & X & \cdot &\cdot & X & X &\cdot
&\cdot& X &X&0 \\
\phantom{\sss{1,2}} & \phantom{\cdot} &\phantom{\cdot}&\phantom{\sss{1,n}} &\phantom{\sss{1,n+1}} & \phantom{\sss{2,1}}&\phantom{\cdot} &\phantom{\sss{2,n}}&\phantom{\sss{2,n+1}}&\phantom{\cdot}&\phantom{\cdot}
&\phantom{\sss{n-1,n}}&\phantom{\sss{n-1,n+1}}&\phantom{\sss{n,n+1}}\\
0 & \cdot&0 & 1 & 0 & 0 &\cdot & 1 & 0 &\cdot&0
& 1 &0&1 \\
0 & \cdot&0 & 0 & 1 & 0 &\cdot & 0 & 1 &\cdot&0
& 0 &1&1
\end{array}\right)
\end{array}
$$
Indeed, consider an arbitrary edge $e$ of the tree $H$. Notice that the boundary vertices of $H$ having the numbers $i$ and $n$, $1\le i\le n-1$, belong to distinct elements of the partition $\cP_{H}(e)$, if and only if the boundary vertices of the tree $G$ with the numbers $i$ and $n$, and also its boundary vertices with the numbers $i$ and $n+1$ belong to distinct elements of the partition $\cP_G(e)$. Therefore, the elements of the row of the matrix $C(G)$ corresponding to the edge $e$ with the numbers of the form $(i,n)$, and $(i,n+1)$ coincide with the elements of the row of the matrix $C(H)$ corresponding to the edge $e$ with the numbers of the form $(i,n)$, $i=1,\ldots,n-1$. Further, the boundary vertices of the tree $G$ having the numbers $n$ and $n+1$ and corresponding to the moustaches chosen always are in the same element of the partition $\cP_G(e)$, and hence the last element of the row corresponding to the edge $e$ in the matrix $C(G)$ equals zero.

In other words, if we eliminate two last rows (that correspond to moustaches chosen) in the matrix $C(G)$, then the resulting matrix can be obtained from the matrix $C(H)$ by doubling the columns with the numbers of the form $(i,n)$, $i=1,\ldots,n-1$, and adding the last zero column.

Further, let the edge $a$ of the moustaches chosen corresponds to the $n$th boundary vertex of the tree  $G$, and the edge $b$ corresponds to the $(n+1)$th one. Then
$$
\cP_G(a)=\{1,\ldots,n-1,n+1\}\sqcup\{n\},\qquad\text{and}\qquad
\cP_G(b)=\{1,\ldots,n-1,n\}\sqcup\{n+1\},
$$
and therefore the last two rows of the matrix $C(G)$ have the form
$$
\begin{array}{cccccccccccccccc}
{\sss{1,2}} & {\cdots} &{\cdots}&{\sss{1,n}} &{\sss{1,n+1}} & {\sss{2,1}}&{\cdots} &{\sss{2,n}}&{\sss{2,n+1}}&{\cdots}&{\cdots}
&{\sss{n-1,n}}&{\sss{n-1,n+1}}&{\sss{n,n+1}}\\
0 & \cdots&0 & 1 & 0 & 0 &\cdots & 1 & 0 &\cdots&0
& 1 &0&1 \\
0 & \cdots&0 & 0 & 1 & 0 &\cdots & 0 & 1 &\cdots&0
& 0 &1&1
\end{array}
$$

Now, show that the rows of the matrix $C(G)$ are linearly independent. Assume that some their linear combination vanishes. Due to inductive assumptions the rows of the matrix $C(H)$ are linear independent, therefore in this combination the last two rows can not have zero coefficients. Then their coefficients are opposite (because the components of these two rows standing at the last column coincides, but the remaining components of the last row equal zero). The latter is impossible, because, on the one hand, any linear combination of the remaining rows of the matrix  $C(G)$ has the same components at the columns with the numbers of the form  $(i,n)$, and $(i,n+1)$, $i=1,\ldots,n-1$, and on the other hand, two last rows taken with opposite coefficients give at those positions non-zero but opposite contributions. Thus, the rows of the matrix $C(G)$ are linear independent, and its rank is equal to the number of its rows, that, in its turn, is equal  to the number of edges of a binary tree with $n$ boundary vertices. Lemma is proved.
\end{proof}

Now recall the definitions introduced by A.~Eremin~\cite{Trudy} and~\cite{Eremin}. Let $S$ be a finite set consisting of $n$ elements. A  \emph{multi-cyclic order of multiplicity $k$ on the set  $S$} is a mapping $\pi\:\Z_{nk}\to S$ such that
\begin{enumerate}
\item $\pi(j)\ne\pi(j+1)$ for any $j\in\Z_{nk}$.
\item  The pre-image of each element $s\in S$ under the mapping $\pi$ consists of exactly $k$ elements.
\end{enumerate}

Let $G$ be a binary tree connecting $M$. A multi-cyclic order on $M$ is said to be {\em matched with the tree $G$} or is referred as a \emph{multi-tour of the tree $G$}, if there exists a positive integer $m$ such that for any $e\in E$ and $M_i\in \cP_G(e)$ there exists exactly $m$ elements $p\in\Z_{nk}$ such that $\pi(p)\in M_i$, but $\pi(p+1)\notin M_i$. This number $m$ is called the \textit{multiplicity of the multi-tour}. By $\mathcal T(G)$  we denote the set of all multi-tours of the tree $G$. It is not difficult to verify, see~\cite{Eremin}, that the following result holds.

\begin{ass}
If a multi-cyclic order is a multi-tour of a tree, then its multiplicity as of a multi-tour coincides with the one as of a multi-cyclic order.
\end{ass}

Let $G=(V,E)$ be an arbitrary binary tree with the boundary  $M=\{p_i\}$.

\begin{ass} \label{ass:PlanarEquiv}
A multi-cyclic order $\pi$ on the set $M$ is a multi-tour of the tree $G$ of a multiplicity $m$, if and only if for any $e\in E$ exactly $2m$ paths in $G$ of the form $\g\big(\pi(j),\pi(j+1)\big)$, $j\in \Z_{mn}$, passes through $e$.
\end{ass}	

Let $G$ be an arbitrary binary tree with the boundary $M$. Each multi-tour $\pi$ corresponds to a cyclic walk $c_\pi$ in the complete graph $K(M)$, whose consecutive edges have the form $\pi(j)\pi(j+1)$. The cyclic walk $c_\pi$ passes each vertex $p_i\in M$ the same number of times, and this number is equal to the multiplicity $k$ of the multi-tour $\pi$. The value
$$
\frac1{2k}\sum_{j=0}^{nk}\r\big(\pi(j)\pi(j+1)\big),
$$
i.e., the length of the walk $c_\pi$ in the complete graph $K(M)$ in the metric space $M$ divided by $2k$ is called the  \emph{multi-perimeter\/} of the multi-tour $\pi$.

\begin{rk}
In the case of $k=1$ a multi-tour is also called just a tour, and its multi-perimeter is called \emph{half-perimeter}.
\end{rk}

By each multi-tour $\pi$ of the tree $G$ construct a vector $w^{\pi}\in\R^{m_1}$, $m_1=n(n-1)/2$, whose component $w^\pi_{ij}$ corresponding to an edge $ij$ of the graph $K(M)$, $i<j$, is equal to the number of occurrences of the edge $ij$ in the walk $c_\pi$.

\begin{lem}\label{lem:sol_cycles}
Let $G$ be an arbitrary binary tree with the boundary $M$, and let $\pi$ be its arbitrary multi-tour of multiplicity $k$. Then the vector $\dfrac{1}{2k}w^\pi$ satisfies System~$(\ref{eq:dual})$, and hence, belongs to the solutions space $\L_G$ of the dual problem, and the function $H$ at it equals the multi=perimeter of the multi-tour $\pi$.
\end{lem}

\begin{proof}
Indeed, as we have already mentioned above, equations of System~$(\ref{eq:dual})$ correspond to edges of the tree $G$, and the variables (i.e., the columns of the matrices) correspond to the edges of the complete graph $K(M)$. Consider the row corresponding to an edge $e$. All its non-zero elements equal $1$ and stand at the positions corresponding to that edges of $K(M)$, whose ends lie in different components of the partition $\cP_G(e)$ of $M$ generated by the edge $e$ of  $G$. Therefore, scalar product of such row and the vector $w^\pi$ is equal to the number of such edges belonging to the multi-tour $\pi$ (with multiplicities). But this number equals $2k$ in accordance with Assertion~\ref{ass:PlanarEquiv}, that implies the first statement of Lemma. The second statement is evident, because the function $H$ is equal to the linear combination of variables, whose values are equal to the edges multiplicities, with the coefficients that equal distances between the vertices.
\end{proof}

Theorem~3.3 from paper~\cite{Eremin} can be reformulated as follows.

\begin{lem}\label{lem:multi_inverse}
For any non-negative integer solution $\l$ of the equations system $A_{12}^T\l_1=-2kf_2$, where  $k\in\N$, there exists a multi-tour $\pi$ of the tree $G$ of multiplicity $k$, such that $\l=w^\pi$.
\end{lem}

The weight formula for minimal parametric filling obtained by A.~Eremin~\cite[Theorem~4.1]{Eremin} permits to reduce the problem of minimal parametric filling finding to the search of a multi-tour of maximal multi-perimeter, and the multiplicity of such multi-tours is estimated from above by the value $(C_n^2)!$.

\begin{ass}\label{ass:vertices}
Let $\cM=(M,\r)$ be an arbitrary finite metric space, and let $G$ be an arbitrary binary tree connecting  $M$. The weight $\mpf_{-}(\cM,G)$ of minimal parametric filling of the type $G$ of a metric space $\cM$ is equal to he maximal value of the objective function $H=\sum d_{ij}\l_{ij}$ on  the vertices of the polyhedral solution space $\L_G$ of the dual problem defined by System~$(\ref{eq:dual})$. If $\l\in\L^*$ is a vertex which the maximum is attained at, then $\mpf_{-}(\cM,G)$ equals multi-perimeter of the multi-tour that corresponds to the solution $\l$.
\end{ass}

\begin{proof}
Indeed, as it is shown in~\cite{IvaTuzhMF}, a parametric minimal filling of a finite metric space does always exist, i.e., the corresponding Linear Programming Problem is always solvable, and hence, the dual problem is also always solvable. Its solution as a solution of CLPP is always attained at some vertex of the polyhedral set $\L_G$ that is the solutions space. The vertices of  $\L_G$ are described by Assertion~\ref{ass:extreme_points}, and in accordance with this Assertion each vertex $\tilde\l$ is defined as the solution of a linear equations system, whose matrix is a non-degenerate square sub-matrix of the matrix $A_{12}^T$ of size $(2n-3)\x(2n-3)$, where $n$ stands for the number of points in $M$, and the right hand side equals $-f_2$ (the vector consisting of $1$). Due to Cramer's rule, the components of this solution are relations of the determinants of matrices consisting of $0$ and $1$, in particular, they are rational numbers, and the remaining components of the vector $\tilde\l$ equals zero.  Multiplying $\tilde\l$ by doubled common denominator $k$ of its non-zero components, we get the vector $\l$ satisfying the system $A_{12}^T\l_1=-2kf_2$ from Lemma~\ref{lem:multi_inverse}. In accordance with the  same Lemma~\ref{lem:multi_inverse},  the vector $\l$ corresponds to a multi-tour $\pi$. Evidently, the multi-perimeter of this multi-tour equals to the value of the objective function $H$ of the dual problem~(\ref{eq:dual}) at the vector $\tilde\l$, i.e., to the weight of the minimal parametric filling.
\end{proof}

This result permits to improve essentially an estimation on the multiplicity of the multi-tours which the weight of minimal parametric filling can attained at.

\begin{ass}\label{ass:multi_estimate}
Let a finite metric space $M$ consists of $n\ge 3$ points, and let $G$ be a binary tree connecting $M$. The weight of minimal parametric filling of the type $G$ of the space $M$ is attained at a multi-tour of the tree $G$, whose multiplicity does not exceed  $2^{2n-5}$.
\end{ass}

\begin{proof}
The proof of Assertion~\ref{lem:multi_inverse} implies that the multiplicity of a multi-tour, which the weight of minimal parametric filling attained at, does not exceed the maximal minor of the matrix  $A_{12}^T$. Recall that the columns of this matrix are enumerated by the edges of the complete graph on the vertex set $M$, the column $(i,j)$ consists of $2n-3$ elements enumerated by the edges of the tree $G$, and this element equals $1$, if and only if the corresponding edge belongs to the path in $G$ connecting its $i$th and $j$th boundary vertices. Such matrices are sometimes called \emph{paths matrices\/} of the graph  (for a fixed family of paths). In paper~\cite{BR} an estimate on determinants of such matrices is obtained, namely, it is shown that the determinant of the paths matrix of any paths system consisting of $(k-1)$ paths in an arbitrary tree with $k$ vertices does nor exceed $2^{k-1}$. And if the paths system contains a pair of paths starting at a common vertex, then the estimate can be improved to $2^{k-3}$. Our case corresponds to  $k=2n-2$. Moreover, for $n\ge 3$ any system of $2n-3$ paths in $G$ connecting boundary vertices satisfies the latter condition, therefore we obtain a stronger estimate. Assertion is proved.
\end{proof}

\begin{rk}
Apparently the best general estimate on the determinant of a matrix consisting of ones and zeros was obtained in~\cite[Problem~523]{FaddSom}. The estimate~\cite{BR} is better, but it is not exact also, especially in our particular case. The sequence $2^{2n-5}$, $n\ge 3$, starts as $2, 8, 32, 128, 512, \ldots$. But the formulas obtained by the authors~\cite{IvaTuzhMF}  and  B.~Bednov~\cite{Bed} together with the computational results of the present paper, see below, show that the sequence $\{k_n\}$ of maximal multiplicities of the multi-tours corresponding to the vertices of the polyhedrons $\L_G$ for the binary trees $G$ with $n\ge 3$ boundary vertices starts as $1, 1, 1, 2, 2,\dots$. An interesting algebraic problem appears: Is it possible to improve the estimate from~\cite{BR} assuming that the trees under consideration are binary, and the paths connect boundary vertices only.
\end{rk}

\section{Examples}
\markright{\thesection.~Examples}
In the present Section we apply our approach to derive formulas for the weight of minimal parametric filling of a given type, and also to explain existence of examples of open families of finite metric spaces each of which has minimal fillings of several distinct types, constructed by Z.~Ovsyannikov in~\cite{Ovs}.

In all examples listed below the edges of the complete graph on an $n$-element set $M$ are enumerated by the pairs $(i,j)$, $i<j$, and ordered lexicographically. The distances in the space  $M$ are denoted by $d_{ij}$. To be short, the matrix $A_{12}^T$ is denoted by $A$.

\subsection{Four-Point Spaces}
Let $n=4$. Consider the unique binary tree with $4$ vertices of degree $1$, and let the vertices of its moustaches are the pairs of points from $M$ having the numbers $1,\,2$ and $3,\,4$, respectively. Enumerate the boundary edges in accordance with the enumeration of the boundary vertices, and let the unique interior vertex have the number $5$. Then
$$
A=\left(
\begin{array}{cccccc}
 1 & 1 & 1 & 0 & 0 & 0 \\
 1 & 0 & 0 & 1 & 1 & 0 \\
 0 & 1 & 0 & 1 & 0 & 1 \\
 0 & 0 & 1 & 0 & 1 & 1 \\
 0 & 1 & 1 & 1 & 1 & 0 \\
\end{array}
\right),
$$
the solution space $\L$ is a straight segment in the $6$-dimensional space with the end-points  (vertices)
$$
\frac{1}{2}\left(1,0,1,1,0,1\right),\qquad
\frac{1}{2}\left(1,1,0,0,1,1\right),
$$
the values of the objective function at this vertices equal
$$
\frac{1}{2}\big(d_{12}+d_{14}+d_{23}+d_{34}\big),\qquad \frac{1}{2}\big(d_{12}+d_{13}+d_{24}+d_{34}\big),
$$
respectively, and the weight of the minimal parametric filling of this type equals to the maximum of these two values  (two half-perimeters of the corresponding tours). This formula is obtained by the authors in~\cite{IvaTuzhMF}.

\subsection{Five-Point Spaces}
Let $n=5$. Consider unique binary tree with $5$ vertices of degree $1$, and let the vertices of its moustaches are the pairs of the points of the space $M$ having the numbers $1,\,2$ and $4,\,5$, respectively, and the remaining boundary vertex has the number $3$. Enumerate the boundary edges in accordance with the enumeration of the boundary vertices, and let the interior edge adjacent with the edges of the moustaches  $1,\,2$ have the number $6$, and the remaining interior edge have the number $7$. Then
$$
A=\left(
\begin{array}{cccccccccc}
 1 & 1 & 1 & 1 & 0 & 0 & 0 & 0 & 0 & 0 \\
 1 & 0 & 0 & 0 & 1 & 1 & 1 & 0 & 0 & 0 \\
 0 & 1 & 0 & 0 & 1 & 0 & 0 & 1 & 1 & 0 \\
 0 & 0 & 1 & 0 & 0 & 1 & 0 & 1 & 0 & 1 \\
 0 & 0 & 0 & 1 & 0 & 0 & 1 & 0 & 1 & 1 \\
 0 & 1 & 1 & 1 & 1 & 1 & 1 & 0 & 0 & 0 \\
 0 & 0 & 1 & 1 & 0 & 1 & 1 & 1 & 1 & 0 \\
\end{array}
\right),
$$
the solution space $\L$ is a $3$-dimensional tetrahedron in $10$-dimensional space with the vertices
\begin{gather*}
\frac{1}{2}\big(1,0,0,1,1,0,0,1,0,1\big),\quad
\frac{1}{2}\big(1,1,0,0,0,0,1,1,0,1\big),\\
\frac{1}{2}\big(1,0,1,0,1,0,0,0,1,1\big),\quad
\frac{1}{2}\big(1,1,0,0,0,1,0,0,1,1\big),
\end{gather*}
and the values of the objective function at these vertices equal
\begin{gather*}
\frac{1}{2}\big(d_{12}+d_{15}+d_{23}+d_{34}+d_{45}\big),\quad
\frac{1}{2}\big(d_{12}+d_{13}+d_{25}+d_{34}+d_{45}\big),\\
\frac{1}{2}\big(d_{12}+d_{14}+d_{23}+d_{35}+d_{45}\big),\quad
\frac{1}{2}\big(d_{12}+d_{13}+d_{24}+d_{35}+d_{45}\big),
\end{gather*}
respectively. The weight of minimal parametric filling of this type equals to the maximum of those four expressions  (the half-perimeters of the corresponding tours depicted in Figure~\ref{fig:five}). This formula can be easily obtained from results of B.~Bednov~\cite{Bed}, and its particular cases were obtained by students of Mechanical and Mathematical Faculty of Lomonosov Moscow Sate University E.~Zaval'nuk, Z.~Ovsyannikov, O.~Rubleva in their term theses.

\ig{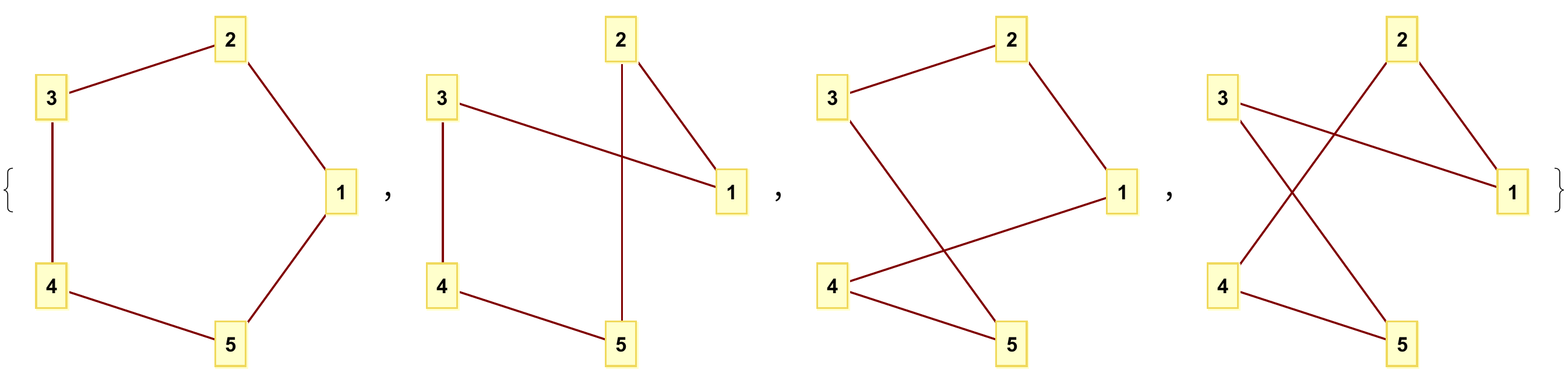}{0.15}{fig:five}{Tours of five-point spaces appearing in the weight formula.}

\subsection{Six-Point Spaces}
Let $n=6$. In this case there are two binary trees with $6$ vertices of degree $1$, one of them has two moustaches, and the other one has three moustaches. In Figure~\ref{fig:six_trees} the both trees are shown together with enumerations of the boundary vertices. The boundary edges are enumerated in accordance with the enumeration of the boundary vertices, and the remaining edges are enumerated as it is shown in Figure~\ref{fig:six_trees}. At first we consider the tree with two moustaches. Then
$$
A=
\left(
\begin{array}{ccccccccccccccc}
 1 & 1 & 1 & 1 & 1 & 0 & 0 & 0 & 0 & 0 & 0 & 0 & 0 & 0 & 0 \\
 1 & 0 & 0 & 0 & 0 & 1 & 1 & 1 & 1 & 0 & 0 & 0 & 0 & 0 & 0 \\
 0 & 1 & 0 & 0 & 0 & 1 & 0 & 0 & 0 & 1 & 1 & 1 & 0 & 0 & 0 \\
 0 & 0 & 1 & 0 & 0 & 0 & 1 & 0 & 0 & 1 & 0 & 0 & 1 & 1 & 0 \\
 0 & 0 & 0 & 1 & 0 & 0 & 0 & 1 & 0 & 0 & 1 & 0 & 1 & 0 & 1 \\
 0 & 0 & 0 & 0 & 1 & 0 & 0 & 0 & 1 & 0 & 0 & 1 & 0 & 1 & 1 \\
 0 & 1 & 1 & 1 & 1 & 1 & 1 & 1 & 1 & 0 & 0 & 0 & 0 & 0 & 0 \\
 0 & 0 & 1 & 1 & 1 & 0 & 1 & 1 & 1 & 1 & 1 & 1 & 0 & 0 & 0 \\
 0 & 0 & 0 & 1 & 1 & 0 & 0 & 1 & 1 & 0 & 1 & 1 & 1 & 1 & 0 \\
\end{array}
\right).
$$

\ig{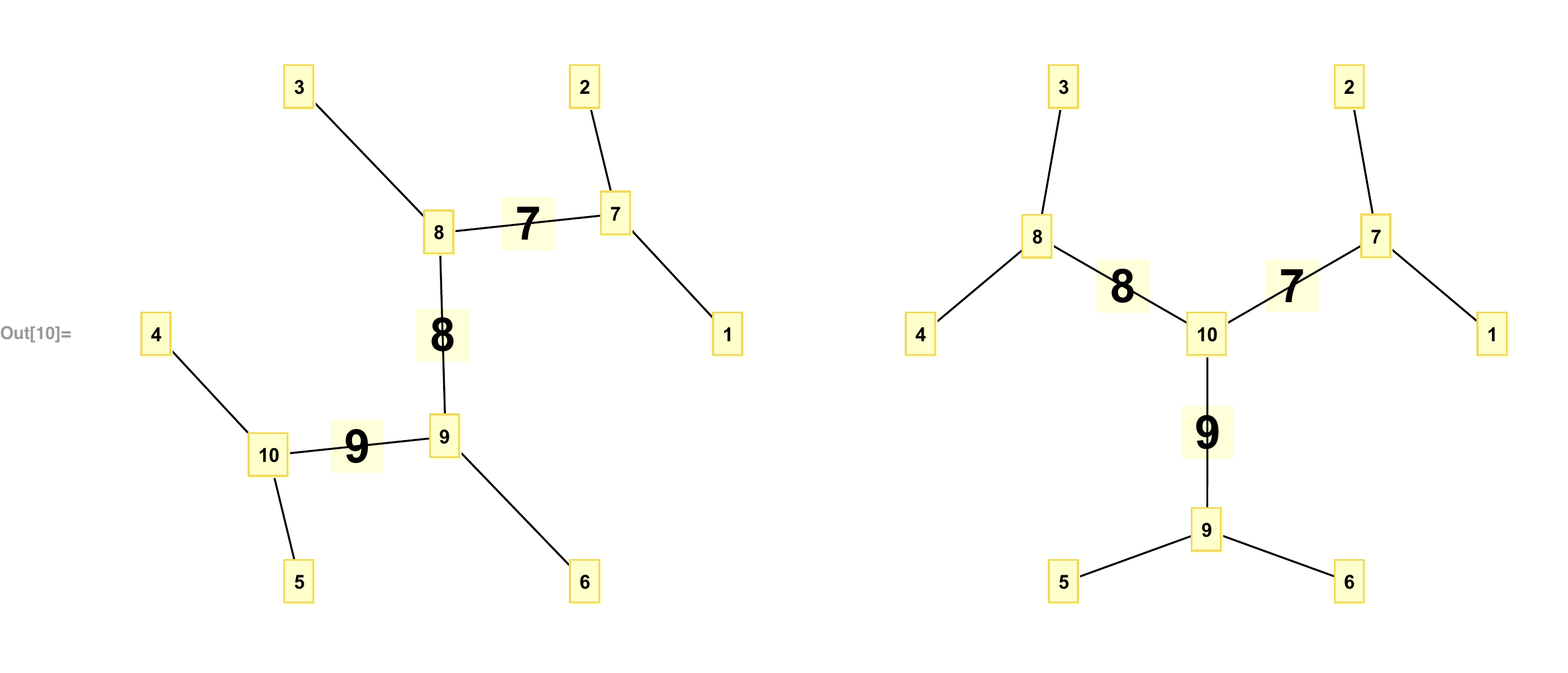}{0.25}{fig:six_trees}{Binary trees with six boundary vertices.}

The solution space $\L$ is a $6$-dimensional convex polyhedron in the $15$-dimensional space. It has $8$ vertices, whose coordinates are
\begin{gather*}
\frac{1}{2}\big(1,0,0,0,1,1,0,0,0,1,0,0,1,0,1\big),\quad
\frac{1}{2}\big(1,1,0,0,0,0,0,0,1,1,0,0,1,0,1\big),\\
\frac{1}{2}\big(1,0,1,0,0,1,0,0,0,0,0,1,1,0,1\big), \quad
\frac{1}{2}\big(1,1,0,0,0,0,1,0,0,0,0,1,1,0,1\big),\\
\frac{1}{2}\big(1,0,0,1,0,1,0,0,0,1,0,0,0,1,1\big),\quad
\frac{1}{2}\big(1,1,0,0,0,0,0,1,0,1,0,0,0,1,1\big),\\
\frac{1}{2}\big(1,0,1,0,0,1,0,0,0,0,1,0,0,1,1\big),\quad
\frac{1}{2}\big(1,1,0,0,0,0,1,0,0,0,1,0,0,1,1\big).
\end{gather*}
The values of the objective function at this vertices are equal to
\begin{gather*}
\frac{1}{2}\big(d_{12}+d_{16}+d_{23}+d_{34}+d_{45}+d_{56}\big),\quad
\frac{1}{2}\big(d_{12}+d_{13}+d_{26}+d_{34}+d_{45}+d_{56}\big),\\
\frac{1}{2}\big(d_{12}+d_{14}+d_{23}+d_{36}+d_{45}+d_{56}\big),\quad
\frac{1}{2}\big(d_{12}+d_{13}+d_{24}+d_{36}+d_{45}+d_{56}\big),\\
\frac{1}{2}\big(d_{12}+d_{15}+d_{23}+d_{34}+d_{46}+d_{56}\big),\quad
\frac{1}{2}\big(d_{12}+d_{13}+d_{25}+d_{34}+d_{46}+d_{56}\big),\\
\frac{1}{2}\big(d_{12}+d_{14}+d_{23}+d_{35}+d_{46}+d_{56}\big),\quad
\frac{1}{2}\big(d_{12}+d_{13}+d_{24}+d_{35}+d_{46}+d_{56}\big),
\end{gather*}
respectively, and the weight of the minimal parametric filling of the type $G$ is equal to the maximum of these eight expressions (eight half-perimeters of the corresponding tours depicted in Figure~\ref{fig:six_sn}).

\ig{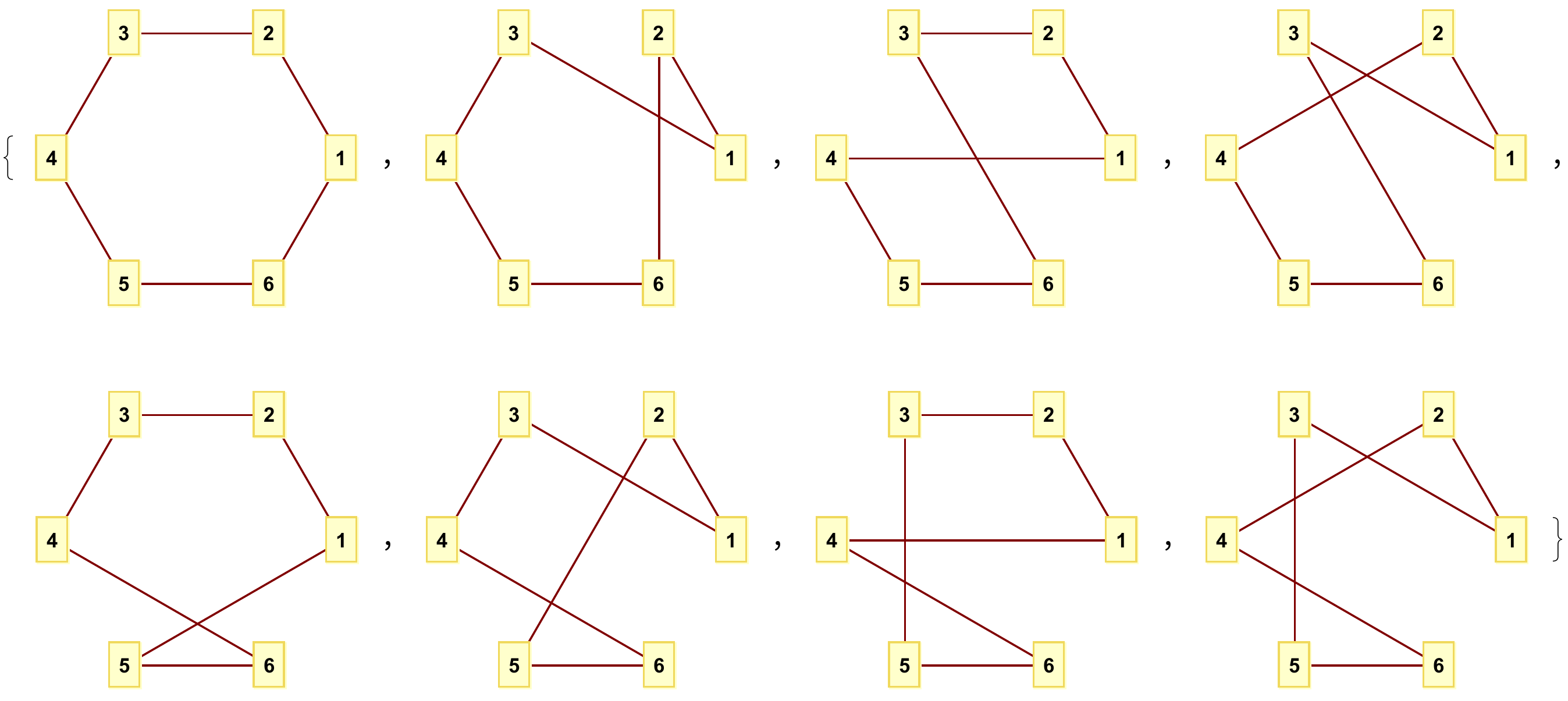}{0.25}{fig:six_sn}{Tours of the binary tree with six boundary vertices and two moustaches that appear in the formula of the weight of minimal parametric filling of this type.}

Let us pass to the case of the tree with three moustaches. We have:
$$
A=
\left(
\begin{array}{ccccccccccccccc}
 1 & 1 & 1 & 1 & 1 & 0 & 0 & 0 & 0 & 0 & 0 & 0 & 0 & 0 & 0 \\
 1 & 0 & 0 & 0 & 0 & 1 & 1 & 1 & 1 & 0 & 0 & 0 & 0 & 0 & 0 \\
 0 & 1 & 0 & 0 & 0 & 1 & 0 & 0 & 0 & 1 & 1 & 1 & 0 & 0 & 0 \\
 0 & 0 & 1 & 0 & 0 & 0 & 1 & 0 & 0 & 1 & 0 & 0 & 1 & 1 & 0 \\
 0 & 0 & 0 & 1 & 0 & 0 & 0 & 1 & 0 & 0 & 1 & 0 & 1 & 0 & 1 \\
 0 & 0 & 0 & 0 & 1 & 0 & 0 & 0 & 1 & 0 & 0 & 1 & 0 & 1 & 1 \\
 0 & 1 & 1 & 1 & 1 & 1 & 1 & 1 & 1 & 0 & 0 & 0 & 0 & 0 & 0 \\
 0 & 1 & 1 & 0 & 0 & 1 & 1 & 0 & 0 & 0 & 1 & 1 & 1 & 1 & 0 \\
 0 & 0 & 0 & 1 & 1 & 0 & 0 & 1 & 1 & 0 & 1 & 1 & 1 & 1 & 0 \\
\end{array}
\right).
$$
The solution space $\L$ is $6$-dimensional convex polyhedron in the $15$-dimensional space. It has  $12$ vertices, whose coordinates are
\begin{gather*}
\frac{1}{2}(1,0,0,0,1,0,1,0,0,1,1,0,0,0,1),\quad
\frac{1}{2}(1,0,1,0,0,0,0,0,1,1,1,0,0,0,1),\\
\frac{1}{2}(1,0,0,1,0,0,1,0,0,1,0,1,0,0,1),\quad
\frac{1}{2}(1,0,1,0,0,0,0,1,0,1,0,1,0,0,1),\\
\frac{1}{4}(2,1,0,0,1,0,1,1,0,2,0,1,1,0,2),\quad
\frac{1}{4}(2,0,1,1,0,1,0,0,1,2,0,1,1,0,2),\\
\frac{1}{2}(1,0,0,0,1,1,0,0,0,1,0,0,1,0,1),\quad
\frac{1}{2}(1,1,0,0,0,0,0,0,1,1,0,0,1,0,1),\\
\frac{1}{4}(2,0,1,0,1,1,0,1,0,2,1,0,0,1,2),\quad
\frac{1}{4}(2,1,0,1,0,0,1,0,1,2,1,0,0,1,2),\\
\frac{1}{2}(1,0,0,1,0,1,0,0,0,1,0,0,0,1,1),\quad
\frac{1}{2}(1,1,0,0,0,0,0,1,0,1,0,0,0,1,1).
\end{gather*}
The values of the objective function at these vertices are
\begin{gather*}
\frac{1}{2}\big(d_{12}+d_{16}+d_{24}+d_{34}+d_{35}+d_{56}\big),\quad
\frac{1}{2}\big(d_{12}+d_{14}+d_{26}+d_{34}+d_{35}+d_{56}\big),\\
\frac{1}{2}\big(d_{12}+d_{15}+d_{24}+d_{34}+d_{36}+d_{56}\big),\quad
\frac{1}{2}\big(d_{12}+d_{14}+d_{25}+d_{34}+d_{36}+d_{56}\big),\\
\frac{1}{4}\big(2d_{12}+d_{13}+d_{16}+d_{24}+d_{25}+2d_{34}+d_{36}+d_{45}+2d_{56}\big),\\
\frac{1}{4}\big(2d_{12}+d_{14}+d_{15}+d_{23}+d_{26}+2d_{34}+d_{36}+d_{45}+2d_{56}\big),\\
\frac{1}{2}\big(d_{12}+d_{16}+d_{23}+d_{34}+d_{45}+d_{56}\big),\quad
\frac{1}{2}\big(d_{12}+d_{13}+d_{26}+d_{34}+d_{45}+d_{56}\big),\\
\frac{1}{4}\big(2d_{12}+d_{14}+d_{16}+d_{23}+d_{25}+2d_{34}+d_{35}+d_{46}+2d_{56}\big),\\
\frac{1}{4}\big(2d_{12}+d_{13}+d_{15}+d_{24}+d_{26}+2d_{34}+d_{35}+d_{45}+2d_{56}\big),\\
\frac{1}{2}\big(d_{12}+d_{15}+d_{23}+d_{34}+d_{46}+d_{56}\big),\quad
\frac{1}{2}\big(d_{12}+d_{13}+d_{25}+d_{34}+d_{46}+d_{56}\big).
\end{gather*}
Notice that now we have four vertices, such that the common denominator of their coordinates equals $4$. They correspond to multi-tours of multiplicity two. The weight of minimal parametric filling of this type $G$ equals to the maximum of those twelve values (of twelve multi-perimeters of the corresponding multi-tours depicted in Figure~\ref{fig:six_tr}).

\ig{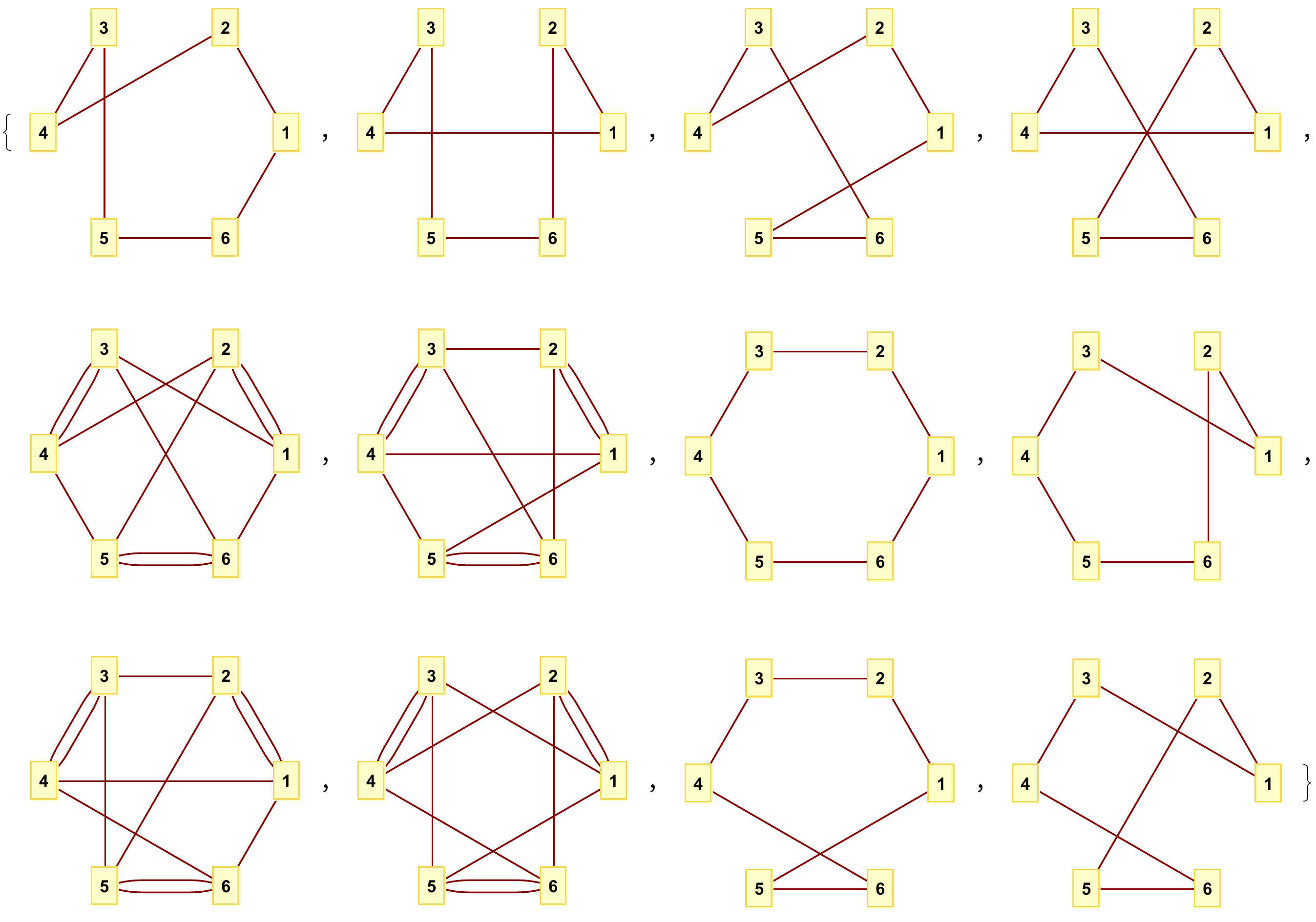}{0.35}{fig:six_tr}{
Multi-tours of the binary tree with six boundary vertices and three moustaches that appear in the formula of the weight of minimal parametric filling of this type.}

The polyhedra for the first and the second cases lie in the same space and have common vertices: $1$st, $2$nd, $5$th, and $6$th vertices of the first one coincide with the $7$th, $8$th, $11$th, and $12$th  vertices of the second one, respectively. This fact explains the appearance of  examples of open families of six-point spaces, each of which has minimal fillings of two different types. Such surprising examples are found firstly by Z.~Ovsysnnikov~\cite{Ovs}. Now their nature becomes clear: the distances in those spaces are chosen in such a way that the objective function attains maximum at a common vertex of both polyhedra. Then for sufficiently close metric spaces the maximum remains at the same vertex.

Generally, the studying of properties of the polyhedra $\L_G$ seems an interesting and perspective problem.

\begin{rk}
For the case of seven-point metric spaces, there are also two binary trees, one with two moustaches, and one with three moustaches. In the both cases we get $10$-dimensional polyhedra in $21$-dimensional space, in the first case the polyhedron has $16$ vertices, and in the second case the polyhedron has $32$ vertices. The multi-tours appear in the second case only, their multiplicities are $2$. We calculated the coordinates of the vertices, so it is easy to write down the explicit formulas for the weights of minimal parametric fillings in this case also as maximum over $16$ and over $32$ expressions, but we do not include them to save the place.
\end{rk}

\begin{rk}
For a fixed binary tree $G$ one can easily write down the matrix $A=C(G)$ and apply Assertion~\ref{ass:extreme_points} to find all the vertices of the polyhedron $\L_G$. The latter permits to write down the explicit formula for the weight of minimal parametric filling of the type $G$ of a metric space as maximum of linear functions on the distances in the space.
\end{rk}

\markright{References}


\begin{thebibliography}{11}

\bibitem{IvaTuzhMF}
A.\,O.~Ivanov and A.\,A.~Tuzhilin,
``One-Dimensional Gromov Minimal Filling Problem'',
Matem. Sbornik, {\bf 203} (5), 65--118  (2012)
[Sbornik: Math., {\bf 203} (5), 677--726  (2012)].

\bibitem{Gromov}
M.~Gromov,
``Filling Riemannian manifolds'',
J. Diff. Geom.,  {\bf 18} (1), 1--147 (1983).

\bibitem{NegativeFillings}
A.\,O.~Ivanov, Z.\,N.~Ovsyannikov, N.\,P.~Strelkova, A.\,A.~Tuzhilin,
``One-Dimensional Minimal Fillings with Negative Edge Weights'',
Vetnik Mosk. Univer., Ser. 1: Matem., Mekh., no.~5,  3--8 (2012).
[Moscow Univ.Math. Bull., {\bf 67} (5--6), 189--194 (2012)].

\bibitem{VasIvanits}
F.\,P.~Vasil'ev, A.\,Yu.~Ivanitskii,
\emph{Linear Programming},
Izd-vo Factorial, Moscow, 1998 [in Russian].

\bibitem{Trudy}
A.\,O.~Ivanov and A.\,A.~Tuzhilin, A.\,Yu.~Eremin,
``Minimal Fillings of Pseudo-metric Spaces'',
in: {\it Trudy Seminara po Vector. i Tenzor. Anaiz i Priloz. k Geom., Mekh. i Fizik.}, v.~27, ed. by A.\,A.~Oshemkov, (Izd-vo Mekh.-Matem. Fac. MGU, Moscow, 2011) pp.~83--105.

\bibitem{Eremin}
A.\,Yu.~Eremin,
``A Formula for the Weight of a Minimal Filling of a Finite Metric Space'',
Matem. Sbornik, {\bf 204} (9), 51--72 (2013)
[Sbornik: Mathematics, {\bf 204} 9, pp.~1285--1306 (2013)].

\bibitem{IT_Musin}
A.\,O.~Ivanov and A.\,A.~Tuzhilin,
``Minimal Fillings of Finite Metric Spaces: The State of the Art'', in: {\it Discrete Geometry and Algebraic Combinatorics}, Vol. 625 of Contemporary Mathematics, ed. by A.~Barg and O.~Musin (AMS, Providence RI, 2014) pp.~9--35.

\bibitem{FaddSom}
D.\,K.~Faddeev, I.\,S.~Sominsky,
{\it Problems in Higher Algebra}
(Nauka, Moscow, 1977; Mir, Moscow, 1972).

\bibitem{BR}
H.~Bruhn, D~Rautenbach.
``Maximal Determinants of Combinatorial Matrices''
 Linear Algebra and Its Appl., {\bf 553}, 37--57 (2017)
 [arXiv:1711.09935v1 (2017)].

 \bibitem{Bed}
B.\,B.~Bednov,
``The Length of Minimal Filling for a Five-Point Metric Space'',
Vetnik Mosk. Univer., Ser. 1: Matem., Mekh., no.~6,  3--8 (2017)
[Moscow University Math. Bull., {\bf 72} (6),  221--225 (2017)].

\bibitem{Ovs}
Z.\,N.~Ovsyannikov,
``An Open Family of Sets That Have Several Minimal Fillings'',
Fund. i Prikl. Matem., {\bf 18} (2), 153--156 (2013)
[J. of Math. Sci., {\bf 203} (6), 855--857 (2014)].

\end{thebibliography}
\end{document}